\newif\ifTwoColumn
\TwoColumnfalse
\ifTwoColumn
	\PassOptionsToClass{reprint}{revtex4-2}
\else
	\PassOptionsToClass{preprint}{revtex4-2}
\fi
\documentclass[aps,pre,nofootinbib,superscriptaddress]{revtex4-2}

\usepackage{amsmath,amssymb,amsthm,bm}
\usepackage{graphicx} 
\usepackage{physics}
\usepackage{booktabs}
\usepackage{microtype}
\usepackage{xcolor}
\usepackage{algorithmicx}
\usepackage{algpseudocode}

\usepackage{hyperref}

\newtheorem{theorem}{Theorem}

\begin{document}
	
\title{
		Finite-Time Observability of Oscillatory Instabilities in Synchronous p-bit Dynamics
	}
	
\author{Naoya Onizawa}
\email{naoya.onizawa.a7@tohoku.ac.jp}
\affiliation{Research Institute of Electrical Communication, Tohoku University, Japan}

\author{Shunsuke Koshita}
\affiliation{Hachinohe Institute of Technology, Japan}

\author{Takahiro Hanyu}
\affiliation{Research Institute of Electrical Communication, Tohoku University, Japan}

\begin{abstract}
	Synchronous update schemes in p-bit annealing offer a natural route to massive parallelism, but they can also induce period-2 oscillations that degrade optimization performance. In practical solvers, such oscillations matter only if they become observable within the finite runtime of the device or simulation. However, most existing analyses are formulated in terms of asymptotic stability and therefore do not directly address when oscillatory modes appear during finite-duration annealing.
	Here we develop a finite-time observability framework for synchronous tick-random p-bit dynamics. Starting from a linearized mean-field description, we derive a graph-dependent criterion that predicts whether unstable modes amplify sufficiently within a finite observation window to produce visible signatures in quantities such as the one-step autocorrelation and the energy trace. This shifts the analysis from asymptotic instability to practical detectability and yields a principled estimate of the reduction in synchrony required to suppress oscillations.
		 We validate the framework on benchmark Max-Cut graph instances and illustrative graph families.  The predicted thresholds capture both the graph dependence of oscillation onset and the finite-time conditions under which oscillations become observable. These results provide a graph-aware basis for selecting the update probability in synchronous p-bit annealers without relying on exhaustive instance-by-instance parameter sweeps.
	
\end{abstract}
	
	\maketitle
	
	\section{Introduction}
	
Synchronous update rules in stochastic spin systems provide an attractive route toward massive parallelism, yet they are well known to induce dynamical instabilities absent in classical single-spin dynamics.\cite{Glauber1963,Peretto1984,Little1974,Hopfield1982}
Previous analyses have primarily characterized such instabilities through asymptotic stability, stationary behavior, or infinite-time bifurcation structure.
In contrast, practical annealing solvers---both numerical and hardware-based---operate over a finite time horizon, within which transient amplification rather than asymptotic divergence determines whether oscillations become observable.\cite{Trefethen1993,Schmid2007}
	
In this work, we introduce a finite-time observability framework for synchronous stochastic spin dynamics and show that period-2 oscillations in tick-random p-bit annealing arise from transiently amplified unstable modes.\cite{Reddy1993,FarrellIoannou1996}
By explicitly incorporating the observation horizon into the stability analysis, we derive graph-dependent conditions under which oscillations are observable or effectively suppressed, even when the underlying dynamics remain asymptotically unstable.
This is the main conceptual distinction from asymptotic analyses: our goal is not to determine whether instability exists in principle at long times, but whether it becomes detectable within the finite runtime relevant to annealing experiments.
	
Ising and p-bit-based annealing methods have emerged as promising approaches
for solving combinatorial optimization problems and sampling from complex distributions.\cite{Camsari2017,Sutton2017,Borders2019,Glover2018QUBO,Barahona1988,BiqMacLibrary}
 In the Max-Cut problem considered below, the objective is to partition the vertices of a weighted graph so that the total weight of edges crossing the partition is maximized; the Max-Cut score denotes this cut value. The G-set instances used in the numerical benchmarks are standard weighted Max-Cut benchmark graphs distributed through the Biq Mac benchmark library.\cite{BiqMacLibrary,Barahona1988,Glover2018QUBO} 
Recent hardware-oriented implementations emphasize parallel updates
to maximize throughput, including synchronous or partially synchronous update schemes.\cite{King2015}
Despite their computational advantages,
synchronous updates are known to alter
the dynamical behavior of stochastic spin systems, relative to classical single-spin dynamics.\cite{Glauber1963,Little1974,Hopfield1982}
	
	In this work, we investigate oscillatory instabilities
	observed in synchronous tick-random p-bit annealing.
	Empirically, we find that full synchrony ($c=1$)
	often leads to persistent period-2 oscillations,
	particularly during the low-temperature phase of annealing.
	Here, $c$ controls the degree of partial synchrony (larger $c$ means fewer simultaneous updates per tick).
	These oscillations can prevent the system from converging
	toward low-energy states.
	Increasing a single control parameter $c$,
	which reduces update simultaneity,
	systematically suppresses oscillations
	and restores annealing performance.
	
	We show that oscillations arise from transient amplification
	of unstable modes in the synchronous mean-field dynamics,
	and that the critical value $c^\star$
	depends on graph-specific spectral properties.
	We therefore formulate the problem in terms of finite-time observability
	rather than asymptotic stability,
	thereby aligning the theory with what is actually measurable in practical annealing implementations.
	
	\paragraph{Contributions}
	\begin{itemize}
		\item We formalize synchronous tick-random p-bit updates
		as a partial-synchrony stochastic process.
		\item We show that full synchrony induces
		period-2 oscillations via transient instability.
		\item We derive graph-dependent oscillation suppression conditions
		based on finite-time mean-field dynamics.
		\item We validate the theory on benchmark graphs
		and provide design rules for choosing $c$.
	\end{itemize}

\section{Model and Update Scheme}

\subsection{Ising energy and p-bit dynamics}

We consider an Ising system composed of $N$ binary stochastic variables (p-bits) $\sigma_i\in\{\pm1\}$, whose energy is defined as
\begin{equation}
	H(\boldsymbol{\sigma}) = -\frac{1}{2}\boldsymbol{\sigma}^\top J \boldsymbol{\sigma} - \boldsymbol{h}^\top \boldsymbol{\sigma},
\end{equation}
where $J$ is a symmetric coupling matrix with zero diagonal and $\boldsymbol{h}$ denotes an external bias field. Throughout this work, we set $\boldsymbol{h}=\boldsymbol{0}$ unless otherwise noted.

For all graph instances considered in this paper, including G-set benchmarks as well as illustrative graphs, we adopt a unified mapping between graph weights $w_{ij}$ and Ising couplings,
\begin{equation}
	J_{ij} = - w_{ij},
\end{equation}
so that minimizing the Ising energy is equivalent to maximizing the corresponding cut value up to an additive constant.

Each p-bit is updated according to a stochastic activation function,
\begin{equation}
	\sigma_i(t+1) = \mathrm{sgn}\!\left[\tanh\!\bigl(I_0 h_i(t)\bigr) - r_i(t)\right],
\end{equation}
where $h_i(t)=\sum_j J_{ij}\sigma_j(t)$ is the local field, $I_0$ controls the effective inverse temperature, and $r_i(t)$ is a uniformly distributed random number in $[-1,1]$.

\subsection{Tick-random synchronous update}

Time is discretized into \emph{ticks}. One tick corresponds to a single global update opportunity for all p-bits. Unlike strictly asynchronous dynamics, where only one spin is updated at a time, or fully synchronous dynamics, where all spins are updated simultaneously, we employ a \emph{tick-random synchronous} update scheme as proposed in Ref.~\cite{Onizawa2024}.

Specifically, at each tick, every p-bit is independently selected for update with probability $p$. If selected, the p-bit updates its state according to Eq.~(3); otherwise, it retains its previous state. All update decisions are made simultaneously at the beginning of each tick, and selected p-bits are updated in parallel. This procedure gives partially synchronous dynamics that interpolate continuously between sparse updates ($p\ll1$) and fully synchronous updates ($p=1$).

We define the parallelism parameter (inverse update probability) as $c=1/p$. Larger values of $c$ correspond to stronger partial deactivation and reduced synchrony.
 For a fixed observation window of $T$ ticks, this convention gives an expected number $pT=T/c$ of update attempts per spin. Thus increasing $c$, or equivalently decreasing $p$, reduces not only simultaneity but also the number of effective spin updates occurring within the same finite observation window. 

\subsection{Finite-time observation and oscillatory behavior}

Because the dynamics are stochastic and the system is typically observed over a finite runtime, the notion of stability in this work is inherently finite-time. Persistent oscillations are identified not by asymptotic divergence but by observable temporal correlations within a given observation window.

We focus on period-2 oscillations, the dominant nontrivial collective mode induced by synchronous updates. As a primary indicator of oscillatory behavior, we measure the one-tick autocorrelation
\begin{equation}
	C(1)=\frac{1}{N}\sum_{i=1}^N \langle \sigma_i(t)\sigma_i(t+1)\rangle_t,
\label{eq:C1_def}
\end{equation}
where $\langle\cdot\rangle_t$ denotes a time average over the observation window. Values of $C(1)$ well below unity indicate frequent spin reversals between successive ticks and serve as a robust indicator of oscillatory dynamics.

In the following sections, we develop a theoretical framework to predict whether such oscillatory modes become observable within a finite number of ticks, depending on the graph structure, the inverse temperature $I_0$, and the parallelism parameter $c$.
	
	\section{Oscillation Phenomenon}
	
	In this section, we characterize the oscillatory behavior observed under synchronous tick-random updates and introduce the operational criterion used in the simulations.
	
	\subsection{Empirical observation}
	
	Under fully synchronous updates ($c=1$), we frequently observe persistent period-2 oscillations in the spin configuration.
	Specifically, a large fraction of spins alternates sign at every tick, leading to a nonconvergent energy trajectory that oscillates between two values.
	Such behavior is robust across random initial conditions and is observed for a wide range of graphs, including G-set benchmarks and illustrative instances.
	
	As the partial deactivation parameter $c$ is increased, the degree of synchrony is reduced and these oscillations are gradually suppressed.
	Beyond a graph-dependent threshold $c^\star$, the dynamics converge toward a stationary regime in which macroscopic observables, such as the energy, fluctuate only due to thermal noise.
	The value of $c^\star$ varies significantly with graph structure, indicating that oscillation suppression cannot be characterized by a universal parameter choice.
	
	Figure~\ref{fig:osc_example} illustrates a representative energy trajectory exhibiting period-2 oscillations at small $c$, as well as their disappearance at larger $c$.
	 The broader $c$-sweep in Fig.~\ref{fig:c_sweep} shows the same suppression trend across the benchmark and illustrative instances used for the threshold analysis. 
	
	\subsection{Definition of observable oscillation}
	
	Because the dynamics are stochastic and simulations are run for a finite number of ticks, oscillations are relevant only insofar as they are observable within a finite time window. In practice, we classify a trajectory as oscillatory when the one-step autocorrelation $C(1)$ falls below a threshold and the energy shows a consistent period-2 alternation over the observation window. In the results below, we use $C(1)<0.5$ as the primary criterion. This finite-time operational definition motivates the transient-growth analysis developed next.
	
	\section{Mean-Field Theory and Transient Stability}
	
	In this section, we develop a finite-time theory for the onset of oscillatory behavior. The analysis proceeds through a controlled modeling hierarchy: stochastic dynamics $\rightarrow$ mean-field map $\rightarrow$ linearization $\rightarrow$ Gaussian effective gain $\rightarrow$ finite-time observability criterion, with an additional  inverse participation ratio (IPR)-based  observability filter when comparing to macroscopic simulation data. The detailed derivation of the mean-field map, the linearization procedure, and the numerical evaluation strategy are given in the Supplemental Material~\cite{SupplementalMaterial} (Secs.~III--VI).  To make the main text self-contained, the central assumptions and physical meanings of these approximations are stated here, while only the algebraic derivations and implementation details are deferred to the Supplemental Material. 
 The IPR is used here as a mode-localization measure: low IPR corresponds to spatially delocalized modes, whereas high IPR corresponds to localized modes. A localized unstable mode can have a large spectral growth rate but leave only a weak signature in global observables, whereas a delocalized mode more readily appears in macroscopic traces such as the energy. 
	
	\subsection{Mean-field dynamical map}
	
	To analyze the collective dynamics, we consider the expectation value $m_i(t)=\mathbb{E}[\sigma_i(t)]$. Under a mean-field approximation, the synchronous tick-random update rule yields the deterministic map
	\begin{equation}
		\bm{m}(t+1)
		= (1-p)\bm{m}(t)
			+ p\tanh\left(I_0(\bm{h}+\bm{J}\bm{m}(t))\right).
			\label{eq:mf_map}
	\end{equation}
	The first term reflects partial deactivation, while the second captures the nonlinear response of the updated p-bits. Equation~\eqref{eq:mf_map} compactly describes how synchrony and interaction strength shape the macroscopic dynamics; a derivation from the stochastic update rule is given in Supplemental Material Sec.~III.
	
	\subsection{Linearization around fixed points}
	
	Linearizing Eq.~\eqref{eq:mf_map} around a fixed point $\bm{m}^\star$ gives $\delta\bm{m}(t+1)=\bm{A}\,\delta\bm{m}(t)$ with Jacobian
	\begin{equation}
		\bm{A}
		= (1-p)\bm{I} + p\,\bm{D}\,I_0\bm{J},
		\label{eq:jacobian}
	\end{equation}
	where $\bm{D}$ is the diagonal gain matrix of the nonlinear response evaluated at $\bm{m}^\star$ (see Supplemental Material Sec.~IV). The eigenmodes of $\bm{A}$ determine how perturbations evolve over successive ticks; period-2 oscillations are associated with modes on the negative real side of the spectrum.
		
		\subsection{Finite-time oscillation suppression}
		
		Classical stability analysis emphasizes the limit $t\to\infty$, whereas practical annealing runs only for a finite number of ticks. We therefore characterize oscillation onset by finite-time amplification. Our criterion is
		\begin{equation}
			\max_{\|\delta\bm{m}(0)\|=1}
			\|\bm{A}^T\delta\bm{m}(0)\|
			< R,
			\label{eq:finite_time}
		\end{equation}
		where $T$ is the observation horizon and $R>1$ is an observability threshold. Equation~\eqref{eq:finite_time} states that no perturbation grows sufficiently within $T$ ticks to produce a detectable oscillatory signature. This shifts the stability question from asymptotic divergence to finite-time observability, the regime relevant to annealing hardware and finite-duration simulations (see Supplemental Material Sec.~V). In the numerical evaluation, this condition is implemented through a reduced spectral proxy rather than by directly computing the full finite-time operator norm; this approximation is summarized in Supplemental Material Sec.~VI.

	\subsection{Gaussian effective gain approximation}
	
	To obtain a tractable graph-level prediction, we replace the state-dependent gain matrix by a scalar effective gain
	\begin{equation}
		\alpha_{\mathrm{eff}}(I_0)
		= \mathbb{E}_{h\sim\mathcal{N}(0,\sigma_h^2)}
		\left[\sech^2(I_0 h)\right],
		\label{eq:alpha_eff}
	\end{equation}
	where $\sigma_h^2$ is the graph-dependent local-field variance. This approximation captures the saturation of the activation function at low temperatures and reduces the Jacobian to a form controlled by the graph spectrum and global parameters. It is expected to be most reliable for sufficiently large, heterogeneous graphs, where the local fields are approximately self-averaging; its accuracy degrades for very small or strongly structured graphs with discrete, non-Gaussian field distributions. Details are given in Supplemental Material Secs.~VII and VIII.
	This approximation follows standard mean-field treatments of disordered Ising models, where the local-field distribution is replaced by a Gaussian surrogate and the diagonal gain is averaged over that distribution.\cite{Mezard1987,Nishimori2001}

	The observability threshold $R$ in Eq.~\eqref{eq:finite_time} has an operational meaning: it is the minimum finite-time amplification required for an unstable mode to leave a visible signature in measured quantities such as the one-step autocorrelation or the energy trace. In this sense, $R$ is not a microscopic fitting parameter but part of the observation model connecting linear growth to experimentally or numerically resolvable oscillations. Throughout this work we use a fixed reference value $R=10$, which corresponds empirically to the onset of clear period-2 energy alternation and to the simulation-side criterion $C(1)<0.5$.

	The main conclusions are not sensitive to the precise choice of $R$. Varying $R$ over a moderate range changes the predicted numerical value of $c^\star$ only modestly and does not alter the qualitative ordering across graphs or the improved agreement obtained from the IPR-corrected theory; see Supplemental Material Sec.~XI and Fig.~S2.
	
\paragraph{Non-IPR vs IPR-corrected theory.}
The baseline theory selects the most unstable mode of $\bm{A}$ and therefore provides a clean graph-dependent stability bound. However, a highly localized unstable mode may leave only a weak signature in global observables such as the energy, even if it is spectrally dominant. We therefore also consider an IPR-corrected variant that emphasizes unstable modes with broader spatial support, since such modes couple more strongly to the macroscopic signals used to identify oscillations.\cite{Anderson1958} In this sense, the IPR correction is not an auxiliary fit, but an observability filter linking linear instability to what can actually be detected in finite-time trajectories. The technical construction is given in Supplemental Material Secs.~IX and X.

	\paragraph{Theory-to-data comparison.}
	 The resulting theoretical stability boundary $c^*(I_0)$ and the quantitative thresholds are compared with simulation in Sec.~V, after the simulation observables have been introduced. 
	
\section{Numerical Experiments}

\subsection{Setup}

Although standard annealing increases $I_0$ from $I_0^{\min}$ to $I_0^{\max}$ over time, we use fixed-$I_0$ simulations when sweeping $c$. This choice is consistent with the theoretical analysis in Section~IV, which assumes time-invariant dynamics over a finite observation horizon. Fixing $I_0$ isolates the effect of partial synchrony and avoids conflating it with transient effects from a time-dependent schedule. We therefore focus on $I_0$ values near $I_0^{\max}$, where oscillatory instabilities are most pronounced in the low-temperature regime, providing a stringent test of oscillation suppression.

We simulate synchronous tick-random annealing for various $c$ values on eight G-set instances and a small set of illustrative graphs. Oscillations are detected using the operational criterion of Section~III, with $C(1)<0.5$ as the primary indicator and the energy trajectory used as a qualitative cross-check. The theory curves use a finite-time horizon $T=40$ and the reference threshold $R=10$, while the simulations use 40 ticks per run.  All $c$-sweeps are therefore compared at a fixed physical observation window of 200 ns, corresponding to 40 ticks in the simulation. Since $p=1/c$ for the tick-random updates used here, this window contains $40/c$ expected update attempts per spin; an equal-update-budget comparison would instead require an observation horizon scaling approximately as $T(c)\propto c$. The thresholds reported below are consequently finite-runtime observability thresholds, not fixed-update-budget thresholds.  Full simulation settings, graph metadata, and the precise temperature normalization are given in Supplemental Material Secs.~XII and XIII. Within the illustrative set, Erd\H{o}s--R\'enyi  (ER)  graphs are included as canonical synthetic random-graph baselines.\cite{Newman2010} Very small illustrative graphs are used mainly as qualitative demonstrations, since mean-field and Gaussian approximations are not expected to be quantitatively accurate in the extreme small-$N$ regime.

	\begin{figure*}[t]
		\centering
		\includegraphics[width=0.90\linewidth]{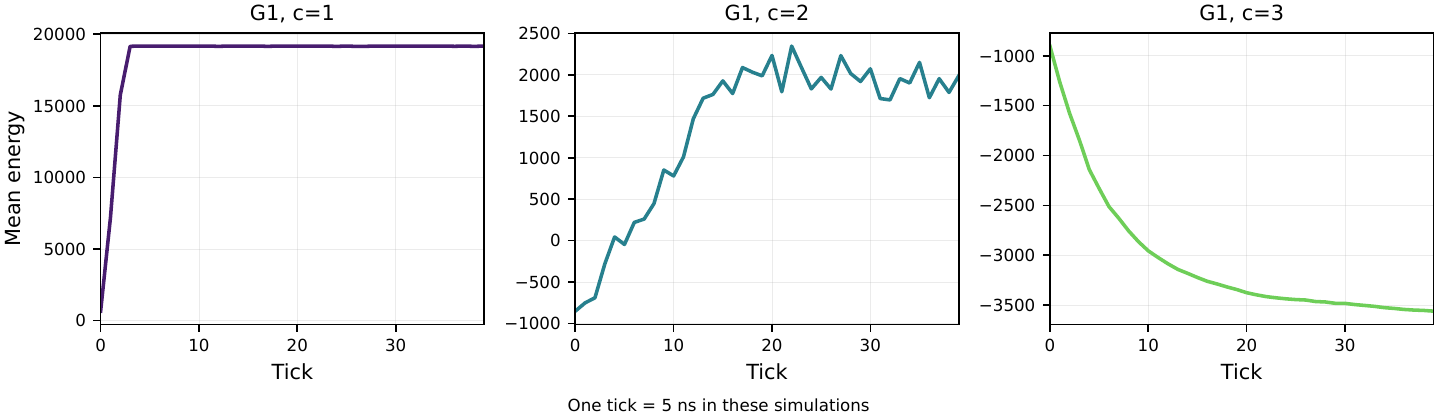}
		\caption{Energy trajectories for G1 at $c=1,2,3$, illustrating suppression of period-2 oscillations as synchrony is relaxed.}
		\label{fig:osc_example}
	\end{figure*}
	
	\begin{figure*}[t]
		\centering
		\begin{minipage}{0.90\linewidth}
			\centering
			\includegraphics[width=\linewidth]{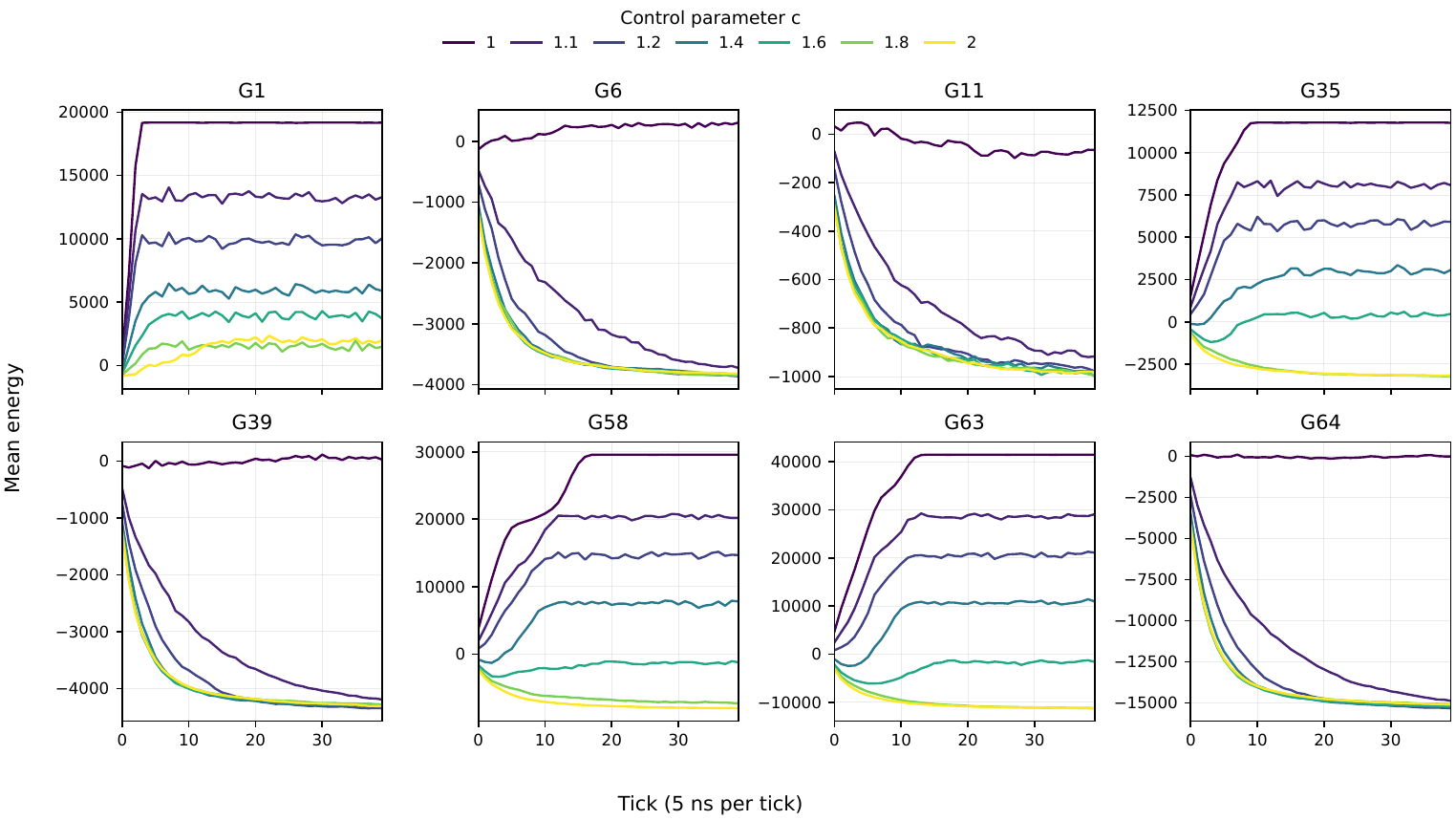}
			\\
			{\small\textbf{(a) G-set}}
		\end{minipage}
		\vspace{1ex}
		\begin{minipage}{0.90\linewidth}
			\centering
			\includegraphics[width=\linewidth]{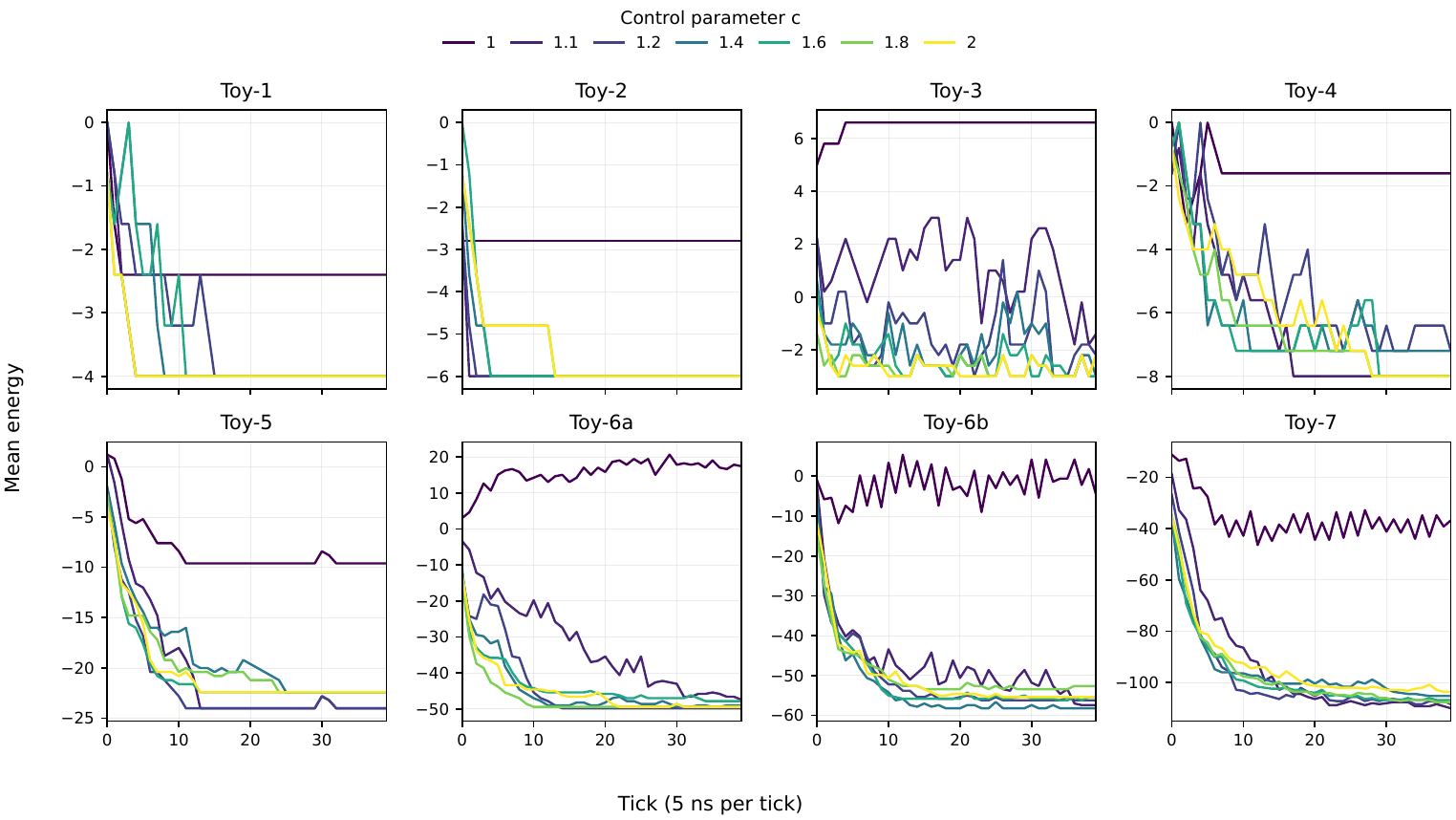}
			\\
			{\small\textbf{(b) Illustrative}}
		\end{minipage}
		\caption{ Energy trajectories for selected low-$c$ values, $c=1,1.1,1.2,1.4,1.6,1.8,2$, showing suppression of oscillations as $c$ increases. }
		\label{fig:c_sweep}
	\end{figure*}
	
	\subsection{Results}
	
	We observe good qualitative and semiquantitative agreement between theory and simulation for the critical behavior once finite-time observability is enforced.
	Across graphs, $c^\star$ increases as the system is driven deeper into the low-temperature regime,
	and the IPR-corrected theory captures the relative ordering of thresholds more consistently
	than the baseline bound.
	Residual discrepancies are attributable to finite-size effects and discretization in the $c$ sweep.
	
	Beyond oscillation suppression itself, the simulations show a practical correlation between dynamical stabilization and optimization quality. Figure~\ref{fig:control_metrics} compares the normalized second-difference amplitude of the late-time energy trace with a normalized Max-Cut score as functions of $c$.  Here, the late-time energy trace denotes the portion of the trajectory after initial transients, from which the oscillation metrics are computed. The normalized Max-Cut score is the cut value obtained in a run divided by the largest observed cut value for the same graph in the corresponding $c$-sweep data, so that the plotted quantity measures relative solution quality rather than graph size. The normalized second-difference amplitude quantifies residual period-2 energy oscillations, whereas the normalized Max-Cut score measures optimization quality; both respond to the stabilization transition, but a one-to-one correspondence is not expected. The one-step autocorrelation $C(1)$ provides a complementary stabilization criterion and need not reproduce the local oscillations observed in the amplitude curve.  In the representative G-set instances shown, the oscillatory regime is characterized by a large second-difference amplitude and poor solution quality. Increasing $c$ drives the system into a low-oscillation regime in which the optimization metric rapidly improves and then saturates. This supports the interpretation that finite-time oscillation suppression is not only a dynamical criterion but also a useful indicator for choosing a parallelism level that preserves solution quality.
	 The corresponding theoretical stability boundaries are shown in Fig.~\ref{fig:theory_boundary}, and the threshold values used for quantitative comparison are listed in Table~\ref{tab:theory_vs_sim}. 
	 The graph-to-graph variation in these control metrics is not explained by sparsity alone: graphs with similar edge density can show different stabilization behavior. The signed and weighted interaction structure, together with the associated spectral characteristics, provides a more plausible explanation for the observed differences. 

	\begin{figure*}[t]
		\centering
		\includegraphics[width=0.96\linewidth]{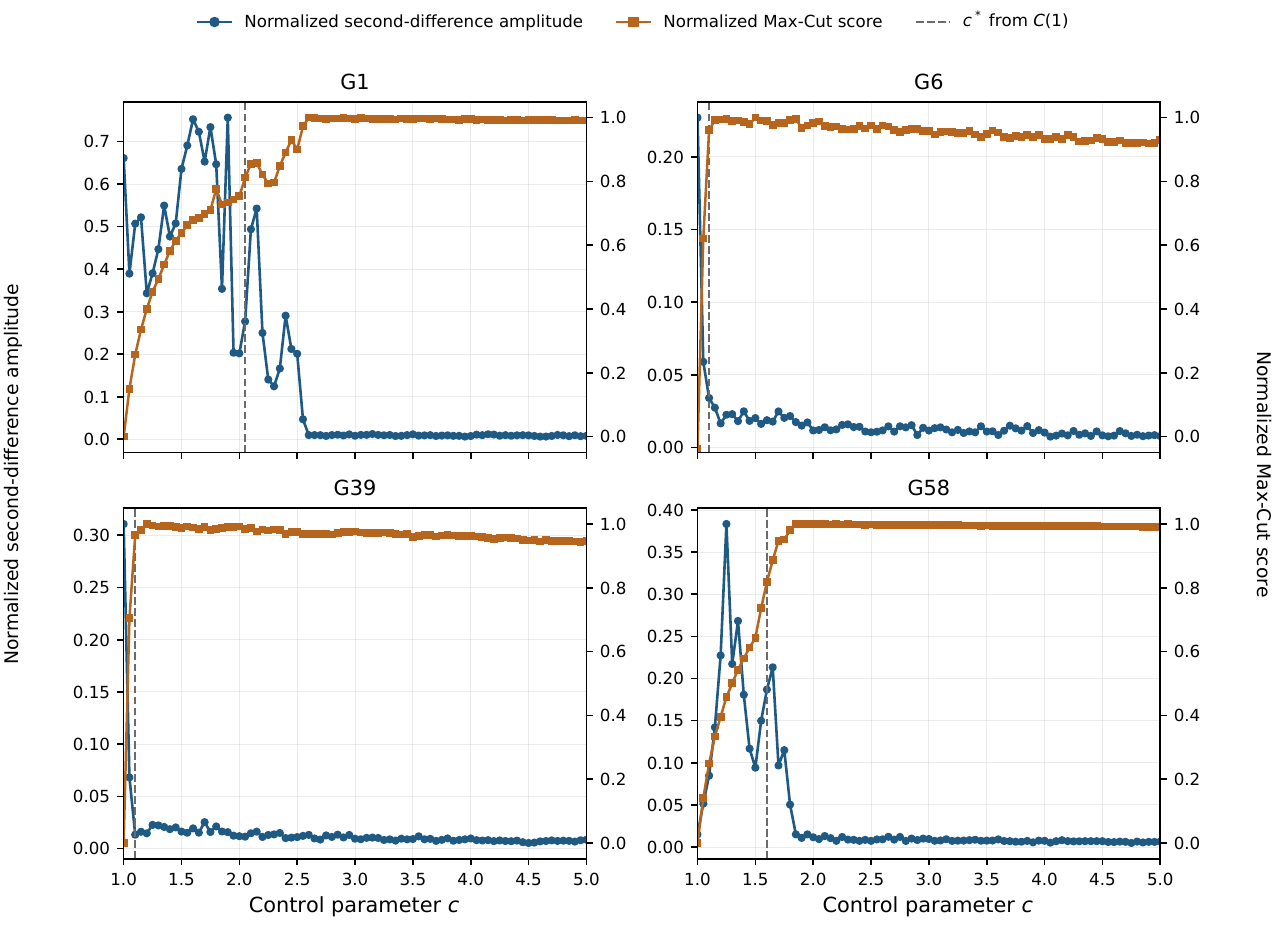}
		\caption{Normalized second-difference amplitude of the late-time energy trace and normalized optimization performance versus the partial-synchrony control parameter $c$ for representative G-set instances. The blue curves show an energy-based oscillation measure obtained from the late-time energy trace, while the orange curves show the normalized Max-Cut score obtained in the same runs. The dashed line marks the simulation threshold $c^\star$ determined from the autocorrelation criterion $C(1)\ge 0.5$. In each case, increasing $c$ suppresses the large-amplitude period-2 energy oscillations and is accompanied by a rapid recovery of optimization performance, after which the score saturates. The figure therefore links finite-time oscillation suppression directly to a practically relevant performance measure.}
		\label{fig:control_metrics}
	\end{figure*}

	\begin{figure*}[t]
		\centering
		\begin{minipage}{0.49\linewidth}
			\centering
			\includegraphics[width=\linewidth]{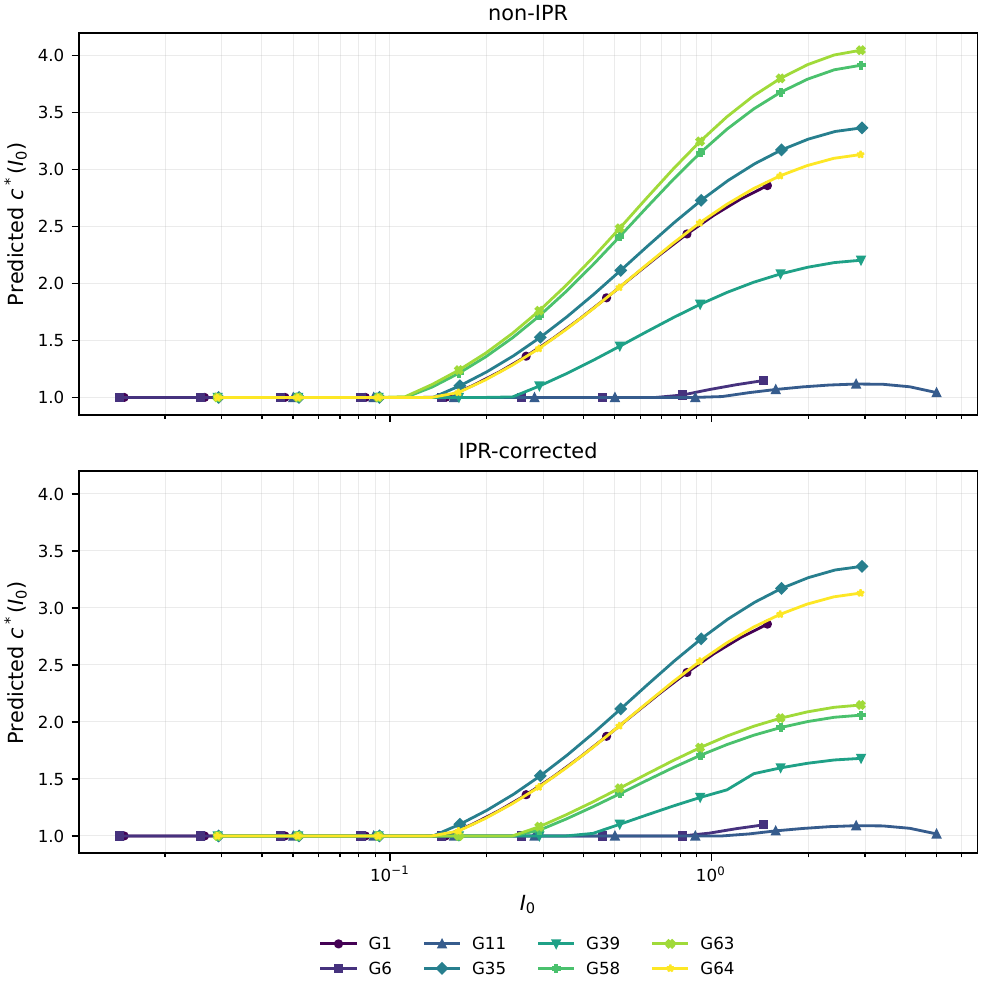}
			\\
			{\small\textbf{(a) G-set}}
		\end{minipage}
		\hfill
		\begin{minipage}{0.49\linewidth}
			\centering
			\includegraphics[width=\linewidth]{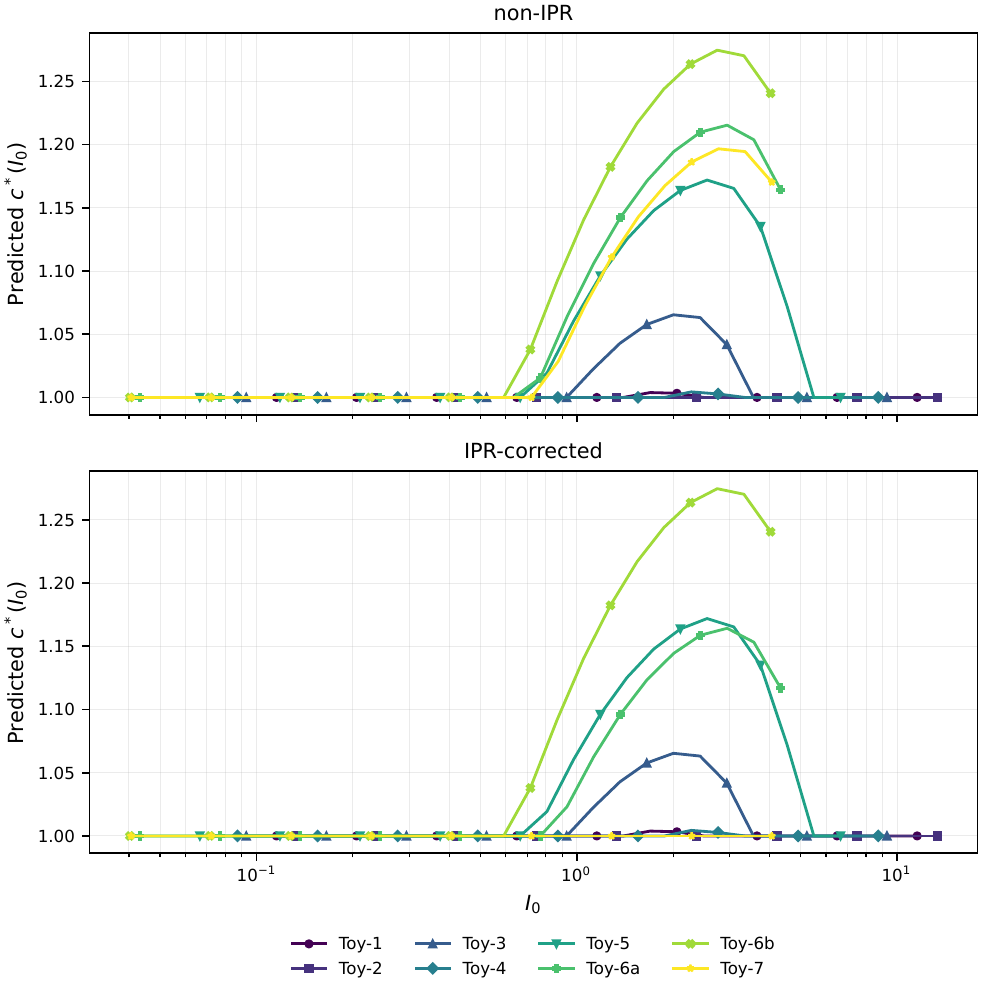}
			\\
			{\small\textbf{(b) Illustrative}}
		\end{minipage}
		\caption{Theoretical stability boundary $c^*(I_0)$ using Gaussian $\alpha_{\mathrm{eff}}(I_0)$ and a finite-time observability criterion. Each panel contains (top) non-IPR and (bottom) IPR-corrected curves, which emphasize unstable but spatially delocalized modes that couple more strongly to macroscopic observables. Small-$N$ illustrative curves are included for qualitative comparison; quantitative agreement improves for larger illustrative ER graphs (Toy-6/7).}
		\label{fig:theory_boundary}
	\end{figure*}
	
	\begin{table}[t]
		\centering
	\caption{Thresholds $c^\star$ from simulation and theory for each graph instance (G-set and illustrative). Theoretical values are evaluated at the nearest $I_0$ to the simulation's $I_0^{\max}$. Simulation thresholds use the $C(1)<0.5$ criterion.}
	\label{tab:theory_vs_sim}
	\begin{tabular}{lcccc}
		\toprule
		Graph & $I_0^{\max}$ (sim) & $c^\star$ (sim) & $c^\star$ (non-IPR) & $c^\star$ (IPR) \\
		\midrule
		G1 & 1.49 & 2.05 & 2.86 & 2.86 \\
		G11 & 5.01 & 1.05 & 1.04 & 1.02 \\
		G35 & 3.12 & 1.65 & 3.36 & 3.36 \\
		G39 & 3.11 & 1.10 & 2.20 & 1.68 \\
		G58 & 3.11 & 1.60 & 3.91 & 2.06 \\
		G6 & 1.45 & 1.10 & 1.15 & 1.10 \\
		G63 & 3.11 & 1.60 & 4.05 & 2.15 \\
		G64 & 3.11 & 1.10 & 3.13 & 3.13 \\
				\midrule
		Toy-1 & 11.55 & 1.20 & 1.00 & 1.00 \\
		Toy-2 & 13.33 & 1.05 & 1.00 & 1.00 \\
		Toy-3 & 9.30 & 1.30 & 1.00 & 1.00 \\
		Toy-4 & 8.73 & 1.05 & 1.00 & 1.00 \\
		Toy-5 & 6.69 & 1.05 & 1.00 & 1.00 \\
		Toy-6a & 4.41 & 1.05 & 1.16 & 1.12 \\
		Toy-6b & 4.09 & 1.10 & 1.24 & 1.24 \\
		Toy-7 & 4.17 & 1.05 & 1.17 & 1.00 \\
		\bottomrule
	\end{tabular}

		\end{table}
	
 For representative small systems with $N<16$, the agreement should be interpreted qualitatively. In this regime the Gaussian closure is less accurate because local-field distributions are discrete and self-averaging is weak; the resulting deviations do not invalidate the finite-time observability framework, but instead mark the limits of the mean-field closure. Representative small-system results are provided in the Supplemental Material. 
	
	\section{Discussion}
	
	The results presented here clarify that oscillations observed in synchronous p-bit annealing are finite-time phenomena rather than direct manifestations of asymptotic instability.
	While classical linear stability analysis focuses on the infinite-time limit, such criteria can substantially overestimate instability in practical annealing processes, where dynamics are observed only over a limited runtime.
	By explicitly incorporating the observation horizon, our analysis describes when oscillatory behavior becomes experimentally or numerically relevant.
	
	A central implication is that the partial-synchrony parameter $c$ should not be viewed merely as an ad hoc oscillation-suppression knob.
	Instead, $c$ acts as a continuous control parameter that mediates a trade-off between massive parallelism and dynamical stability.
	Excessive reduction of synchrony undermines the computational advantage of synchronous updates, whereas insufficient reduction leads to observable oscillations.
	The proposed finite-time observability framework enables a principled, graph-dependent selection of $c$, thereby reducing the need for costly empirical parameter sweeps.
	
	The Gaussian effective-gain approximation plays a key role in rendering the theory tractable and in capturing the stabilizing effect of saturation at low temperatures.
	Although this approximation neglects higher-order correlations, it explains the qualitative trends and, for sufficiently large graphs, the quantitative trends observed in simulations.
Deviations for small illustrative systems highlight the limitations of the mean-field and self-averaging assumptions rather than failures of the underlying dynamical picture.
	
	The introduction of the IPR-corrected criterion further emphasizes the distinction between mathematical instability and observable behavior.
	While the baseline (non-IPR) analysis identifies the most unstable eigenmodes, these modes can be highly localized and thus weakly coupled to macroscopic observables such as the energy.
	By prioritizing delocalized modes, the IPR-corrected variant effectively functions as an observation model, bridging the gap between theoretical predictions and experimentally accessible quantities.
	
	Several extensions of the present work are worth exploring.
	First, incorporating time-dependent annealing schedules $I_0(t)$ into the finite-time observability framework may further improve predictive power in practical settings.
	Second, extending the analysis beyond mean-field approximations could clarify the role of nonlinear effects and correlations in small or highly structured graphs.
	Finally, the present approach may inform the design of hardware p-bit annealers, where finite runtime, limited observability, and parallel update constraints are intrinsic rather than incidental.
	
	Overall, this work frames oscillation suppression in synchronous p-bit annealing as a finite-time, graph-dependent parameter-design problem.
	This perspective provides a theoretical basis for exploiting parallelism while reducing empirical trial-and-error tuning.
	
	\section{Conclusion}
	
	In this work, we developed a finite-time observability framework for synchronous tick-random p-bit annealing and showed that it can predict graph-dependent oscillatory behavior.
	By shifting the focus from asymptotic stability to finite-time observability, we showed that oscillations appear as transient phenomena that become detectable only after sufficient amplification within a limited runtime.
	
	The proposed analysis yields practical, graph-aware design rules for selecting the update probability $p=1/c$, thereby enabling stable annealing dynamics without sacrificing the parallelism advantages of synchronous updates.
Numerical experiments on G-set benchmarks and illustrative graphs show that incorporating finite-time observability and mode delocalization improves agreement between theory and simulation.
	
	Beyond explaining oscillation suppression, the present framework casts parameter tuning in synchronous p-bit annealing as a predictive design problem rather than an empirical trial-and-error process.
	This perspective is particularly relevant for large-scale and hardware-based Ising solvers, where exhaustive parameter sweeps are impractical.
	
	Future work will extend the observability-aware theory to time-dependent annealing schedules and hardware-specific constraints, such as communication delays, finite precision, and device variability.
	We expect that the finite-time observability approach introduced here will provide a useful foundation for the systematic design of scalable and efficient parallel Ising machines.

\paragraph{Code availability.}
The simulation and analysis codes used in this study are publicly available at \url{https://github.com/nonizawa/sync-pbit-theory}.

\begin{acknowledgments}
This work was supported in part by KIOXIA Corporation and by a research grant from the Murata Science and Education Foundation.
\end{acknowledgments}
	
	\clearpage
\setcounter{section}{0}
\setcounter{subsection}{0}
\setcounter{equation}{0}
\setcounter{figure}{0}
\setcounter{table}{0}
\renewcommand{\theequation}{S\arabic{equation}}
\renewcommand{\thefigure}{S\arabic{figure}}
\renewcommand{\thetable}{S\arabic{table}}
\makeatletter
\renewcommand{\theHsection}{supp.\Roman{section}}
\renewcommand{\theHsubsection}{supp.\Roman{section}.\arabic{subsection}}
\renewcommand{\theHequation}{supp.\arabic{equation}}
\renewcommand{\theHfigure}{supp.\arabic{figure}}
\renewcommand{\theHtable}{supp.\arabic{table}}
\makeatother

\section*{Supplemental Material}
\section{Scope of This Supplemental Material}

 This Supplemental Material collects the derivations, implementation details, and compact diagnostic checks that support the main text. It includes the stochastic update pseudocode, the mean-field and finite-time observability derivations, the Gaussian and inverse participation ratio details, sensitivity to the observability threshold $R$, numerical protocol information, graph metadata, and additional diagnostics. 

Throughout this Supplemental Material, all time variables are measured in ticks, and the update probability is denoted by $p=1/c$.

\section{Stochastic Update Rule and Algorithmic Form}

We consider an Ising system of $N$ stochastic binary variables $\sigma_i\in\{\pm1\}$ with energy
\begin{equation}
H(\boldsymbol{\sigma})=
-\frac{1}{2}\boldsymbol{\sigma}^{\top}J\boldsymbol{\sigma}
-\boldsymbol{h}^{\top}\boldsymbol{\sigma},
\label{eqS:energy}
\end{equation}
where $J$ is symmetric with zero diagonal. In the graph instances used in this work, the coupling matrix is obtained from graph weights $w_{ij}$ through
\begin{equation}
J_{ij}=-w_{ij},
\label{eqS:Jmap}
\end{equation}
so that minimizing the Ising energy is equivalent to maximizing the cut value up to an additive constant.

Each selected p-bit is updated according to
\begin{equation}
\sigma_i(t+1)=
\mathrm{sgn}\!\left[\tanh\!\bigl(I_0 h_i(t)\bigr)-r_i(t)\right],
\label{eqS:stochastic_update}
\end{equation}
with local field
\begin{equation}
h_i(t)=\sum_j J_{ij}\sigma_j(t),
\label{eqS:local_field}
\end{equation}
inverse-temperature parameter $I_0$, and a uniform random variable $r_i(t)\in[-1,1]$.

Under tick-random synchronous dynamics, each p-bit is independently selected for update with probability $p=1/c$ at each tick. If selected, it is updated according to Eq.~(\ref{eqS:stochastic_update}); otherwise, it retains its previous state.

The overall update procedure used throughout this work is summarized in Algorithm S1 (Fig.~\ref{figS:algorithm}).

\begin{figure}[t]
\centering
\caption{Pseudocode (Algorithm S1):\ Synchronous tick-random p-bit annealing with partial synchrony parameter $c$.}
\label{figS:algorithm}
\begin{minipage}{0.95\linewidth}
\begin{algorithmic}[1]
\Require Couplings $\bm{J}$, fields $\bm{h}$, schedule $I_0(t)$, parameter $c\ge 1$, horizon $T$
\State Initialize spins $\bm{\sigma}(0)\in\{-1,+1\}^N$
\For{$t=0,1,\dots,T-1$}
\State $p\gets 1/c$
\State Sample mask $\bm{u}(t)\in\{0,1\}^N$ i.i.d. with $\Pr(u_i=1)=p$
\ForAll{$i$ with $u_i(t)=1$ \textbf{in parallel}}
\State $I_i(t)\gets I_0(t)\left(h_i+\sum_j J_{ij}\sigma_j(t)\right)$
\State Draw $\sigma_i(t+1)\in\{-1,+1\}$ with $\Pr(\sigma_i(t+1)=+1)=\frac12[1+\tanh(I_i(t))]$
\EndFor
\ForAll{$i$ with $u_i(t)=0$ \textbf{in parallel}}
\State $\sigma_i(t+1)\gets \sigma_i(t)$
\EndFor
\EndFor
\State \Return $\{\bm{\sigma}(t)\}_{t=0}^{T}$
\end{algorithmic}
\end{minipage}
\end{figure}

In implementation, each tick consists of computing local fields and sampling new states for the active spins. For sparse $J$, the field computation scales as $O(|E|)$ per tick; for dense $J$, it scales as $O(N^2)$.

\section{Mean-Field Reduction}

To describe the collective dynamics, we introduce the single-site expectations
\begin{equation}
m_i(t)=\mathbb{E}[\sigma_i(t)].
\label{eqS:mi_def}
\end{equation}
Under a mean-field closure that neglects higher-order correlations, the update rule yields a deterministic map for $\bm{m}(t)$.

For a selected spin, averaging Eq.~(\ref{eqS:stochastic_update}) over $r_i(t)$ gives
\begin{equation}
\mathbb{E}\!\left[\sigma_i(t+1)\mid h_i(t)\right]
=\tanh\!\bigl(I_0 h_i(t)\bigr),
\label{eqS:conditional_mean}
\end{equation}
because the probability of the next state being $+1$ is
\begin{equation}
\Pr\!\left[\sigma_i(t+1)=+1\mid h_i(t)\right]
=\frac12\left[1+\tanh\!\bigl(I_0 h_i(t)\bigr)\right].
\label{eqS:bernoulli_mean}
\end{equation}

Since a site updates with probability $p$ and is otherwise unchanged, we obtain
\begin{equation}
m_i(t+1)
=(1-p)m_i(t)
+p\,\mathbb{E}\!\left[\tanh\!\left(I_0\sum_j J_{ij}\sigma_j(t)\right)\right].
\label{eqS:pre_mf}
\end{equation}
Approximating the fluctuating field by its mean-field value $\sum_j J_{ij}m_j(t)$ gives
\begin{equation}
\bm{m}(t+1)
=(1-p)\bm{m}(t)
+p\,\tanh\!\left(I_0(\bm{h}+\bm{J}\bm{m}(t))\right),
\label{eqS:mf_map}
\end{equation}
where the hyperbolic tangent is applied componentwise.

Equation~(\ref{eqS:mf_map}) is the deterministic map used in the theoretical analysis. The first term reflects partial deactivation, while the second describes the nonlinear response of updated p-bits.

\section{Linearization and Jacobian Structure}

Let $\bm{m}^\star$ be a fixed point of Eq.~(\ref{eqS:mf_map}). Writing $\delta\bm{m}(t)=\bm{m}(t)-\bm{m}^\star$ and linearizing around $\bm{m}^\star$ yields
\begin{equation}
\delta\bm{m}(t+1)=\bm{A}\,\delta\bm{m}(t),
\label{eqS:linearized_map}
\end{equation}
with Jacobian
\begin{equation}
\bm{A}=(1-p)\bm{I}+p\,\bm{D}\,I_0\bm{J},
\label{eqS:jacobian}
\end{equation}
where
\begin{equation}
\bm{D}
=\mathrm{diag}\!\left(
1-\tanh^2\!\left(I_0(\bm{h}+\bm{J}\bm{m}^\star)\right)
\right).
\label{eqS:Dmatrix}
\end{equation}

The matrix $\bm{D}$ contains the local gains of the nonlinear activation function evaluated at the fixed point. Saturated sites contribute smaller diagonal gain, while unsaturated sites contribute larger gain. The eigenmodes of $\bm{A}$ determine how perturbations evolve over successive ticks.

For the oscillatory behavior of interest here, period-2 modes are associated with eigenvalues on the negative real axis. More generally, observable oscillations arise when perturbations are amplified sufficiently over the finite observation window.

\section{Finite-Time Observability Criterion}

The main text focuses on finite-time observability rather than asymptotic stability. The corresponding criterion is
\begin{theorem}[Finite-time oscillation suppression]
Let $\bm{A}$ be the Jacobian of the mean-field map in Eq.~(\ref{eqS:jacobian}). If for a finite horizon $T$,
\begin{equation}
\max_{\|\delta\bm{m}(0)\|=1}\|\bm{A}^{T}\delta\bm{m}(0)\|<R,
\label{eqS:finite_time}
\end{equation}
for an observability threshold $R>1$, then no observable oscillation emerges within $T$ steps.
\end{theorem}

\begin{proof}
Observable oscillations require transient amplification of perturbations to a level detectable in macroscopic observables such as the energy or one-step autocorrelation. Equation~(\ref{eqS:finite_time}) states that every unit-norm perturbation remains below the observability threshold over the full observation window. Under this condition, no linearly amplified mode reaches an amplitude large enough to produce a detectable period-2 signature within $T$ steps.
\end{proof}

The threshold $R$ should be viewed as part of the observation model rather than as a microscopic constant. It parameterizes the finite-time amplification required for an unstable mode to generate a visible oscillatory signature in measured observables.

\section{Spectral Product Approximation Used in Numerical Evaluation}

For the numerical phase-boundary calculation, we employ a reduced spectral proxy for Eq.~(\ref{eqS:finite_time}). Specifically, the growth factor is approximated as
\begin{equation}
G=\prod_{t=1}^{T}|\lambda_{\min}(t)|,
\label{eqS:growth_factor}
\end{equation}
where $\lambda_{\min}(t)$ denotes the minimum eigenvalue of the Jacobian at tick $t$. For fixed $I_0$, this eigenvalue is time independent and Eq.~(\ref{eqS:growth_factor}) reduces to a power law in $T$.

In practice, the numerical criterion is implemented as
\begin{equation}
\log G \ge \log R.
\label{eqS:log_growth_criterion}
\end{equation}
This approximation replaces the operator norm in Eq.~(\ref{eqS:finite_time}) by a spectral product based on the most unstable oscillatory mode. It is therefore best interpreted as an efficient observability-oriented proxy rather than an exact evaluation of the finite-time norm growth.

\section{Gaussian Effective-Gain Approximation}

The exact Jacobian in Eq.~(\ref{eqS:jacobian}) depends on the fixed point through the state-dependent diagonal matrix $\bm{D}$. To obtain a tractable graph-level prediction, we replace $\bm{D}$ by a scalar effective gain,
\begin{equation}
\alpha_{\mathrm{eff}}(I_0)
=
\mathbb{E}_{h\sim\mathcal{N}(0,\sigma_h^2)}
\left[\sech^2(I_0 h)\right],
\label{eqS:alpha_eff}
\end{equation}
with local-field variance
\begin{equation}
\sigma_h^2=\frac{1}{N}\sum_{i=1}^{N}\sum_{j=1}^{N}J_{ij}^2.
\label{eqS:sigma_h}
\end{equation}

Under this approximation,
\begin{equation}
\bm{A}\approx (1-p)\bm{I}+p\,\alpha_{\mathrm{eff}}(I_0)\,I_0\bm{J}.
\label{eqS:gaussian_jacobian}
\end{equation}

The physical idea is to replace the local fields by a Gaussian random variable with graph-dependent variance. The nonlinear gain $1-\tanh^2(I_0 h)=\sech^2(I_0 h)$ is then averaged over that field distribution. This approximation captures the saturation of the activation function at low temperatures while avoiding the need to compute the exact fixed point.

For sufficiently large, disordered graphs, this treatment is consistent with standard mean-field reasoning, where sums of many heterogeneous couplings generate an approximately Gaussian field distribution. The resulting effective Jacobian depends only on the graph spectrum and global parameters such as $I_0$ and $p$.

\section{Validity Limits of the Gaussian and Self-Averaging Assumptions}

The Gaussian effective-gain approximation is not expected to be uniformly accurate across all graph classes. Its limitations are most visible in very small or highly structured systems, where the local-field distribution is discrete, strongly non-Gaussian, and lacks self-averaging.

This limitation is particularly relevant for the smallest illustrative graphs used in the main text. In these systems, the theoretical boundary can saturate near $c^\star\approx 1$, while simulations may still require slightly larger values of $c$ to suppress oscillations. We interpret this discrepancy as a limitation of the mean-field and Gaussian closures rather than as a failure of the finite-time observability picture itself.

As the system size increases, especially for the larger illustrative Erd\H{o}s--R\'enyi graphs and for the G-set instances, the agreement between the Gaussian theory and simulation improves. This is consistent with the expectation that self-averaging becomes more reliable for larger heterogeneous graphs.

\section{IPR-Corrected Observability Metric}

The baseline theory evaluates instability using the most unstable eigenmode of the Jacobian. However, spatially localized unstable modes may couple only weakly to macroscopic observables such as the total energy. To better align the theory with what is observable in numerical trajectories, we introduce an IPR-based correction.

For a normalized eigenvector $\bm{v}$, the inverse participation ratio is
\begin{equation}
\mathrm{IPR}(\bm{v})=\sum_{i=1}^{N}v_i^4,
\qquad
\sum_{i=1}^{N}v_i^2=1.
\label{eqS:ipr}
\end{equation}
Large IPR indicates a localized mode, while small IPR indicates a delocalized mode.

We define the observability-oriented mode score
\begin{equation}
\mathcal{S}_{\ell}=
\frac{|\lambda_{\ell}|}{\mathrm{IPR}(\bm{v}_{\ell})^{\gamma}},
\label{eqS:ipr_score}
\end{equation}
where $\lambda_{\ell}$ and $\bm{v}_{\ell}$ are the eigenvalue and eigenvector of the candidate mode, and $\gamma>0$ controls the strength of the localization penalty. In this work, we use $\gamma=1$ unless otherwise stated.

The IPR-corrected boundary is then constructed by prioritizing unstable modes with large $\mathcal{S}_{\ell}$ rather than selecting only the mode with the most negative eigenvalue. This makes the theory more sensitive to delocalized unstable modes, which are more likely to produce visible signatures in global observables.

\section{Numerical Mode-Selection Details}

In the numerical evaluation of the IPR-corrected theory, a fixed number of low-lying eigenmodes is inspected for each graph and parameter set. Highly localized modes are filtered through the IPR-based ranking, and the observability condition is evaluated using the best-ranked unstable candidates. These mode-selection settings are fixed across graph instances and are not retuned graph by graph.

The role of the IPR correction is therefore not to fit individual datasets, but to incorporate a simple observation model that distinguishes mathematically unstable modes from modes likely to be visible in a finite-time energy trajectory.

\section{Sensitivity to the Observability Threshold \texorpdfstring{$R$}{R}}

In the main text, we use $R=10$ as a representative observation threshold. This value empirically corresponds to the appearance of clear period-2 energy alternations and to the simulation-side oscillation criterion based on the one-step autocorrelation $C(1)<0.5$.

The value of $R$ is not a first-principles constant. Instead, it parameterizes the transient amplification required to convert an unstable mode into a visible finite-time signal. Varying $R$ therefore shifts the numerical location of the predicted threshold $c^\star$.

In sensitivity checks with $R=5,10,20$, the numerical values of the predicted boundaries change modestly, but the main qualitative trend is unchanged: finite-time observability improves graph-dependent prediction relative to asymptotic reasoning, and the IPR-corrected theory generally remains closer to the simulation thresholds than the baseline non-IPR evaluation.

Figure~\ref{figS:r_robustness} demonstrates that the theoretical prediction is robust with respect to the choice of the observability threshold $R$.

\begin{figure}[t]
\centering
\includegraphics[width=0.98\linewidth]{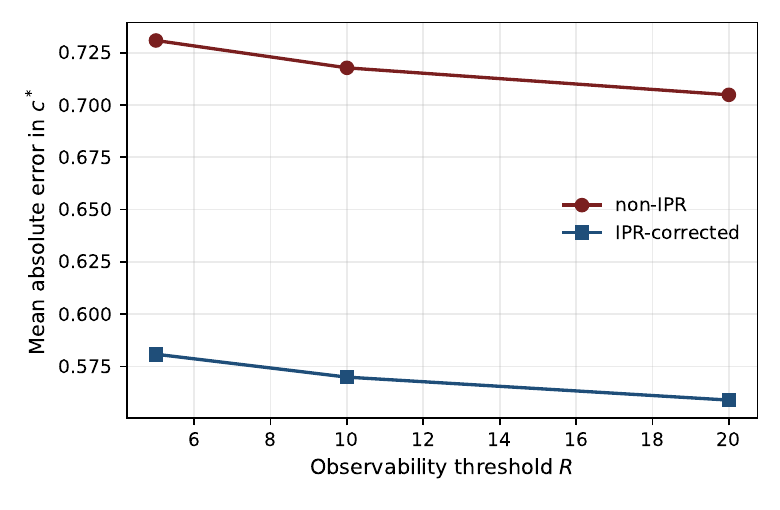}
\caption{Sensitivity of the theoretical threshold prediction to the observability parameter $R$. Shown is the mean absolute error in the predicted $c^\star$ relative to the simulation-derived threshold, averaged over all 16 graph instances used in the comparison (G-set and illustrative graphs), comparing the baseline non-IPR theory and the IPR-corrected theory for $R=5,10,20$. The quantitative error changes only modestly across this range, while the IPR-corrected prediction remains consistently closer to the simulation thresholds. This supports the interpretation of $R$ as an observation-model parameter rather than a finely tuned fit parameter.}
\label{figS:r_robustness}
\end{figure}

\section{Numerical Protocol}

 The simulations in the main text use fixed-$I_0$ sweeps over the partial-synchrony parameter $c$ to isolate the effect of synchrony and align the simulations with the time-invariant finite-time theory. The fixed observation window is 200 ns, corresponding to 40 ticks; with $p=1/c$, this gives $40/c$ expected update attempts per spin. 

 For completeness, the one-step autocorrelation used as the primary oscillation criterion is 
\begin{equation}
C(1)=\frac{1}{N}\sum_{i=1}^{N}\langle \sigma_i(t)\sigma_i(t+1)\rangle_t,
\label{eqS:C1}
\end{equation}
 with the average taken over the observation window. A run is classified as oscillatory when $C(1)<0.5$, with the energy trajectory used as a qualitative consistency check. 

 The normalized second-difference amplitude used in the main text is computed from the post-transient energy trace $H_t=H(\boldsymbol{\sigma}(t))$ as the mean absolute second difference, $\langle |H_{t+1}-2H_t+H_{t-1}| \rangle$, normalized by the late-time energy range. This compactly records residual period-2 energy alternation without repeating the full interpretation given in the main text. 

The graph-dependent inverse-temperature range is set by
\begin{equation}
s_J = \frac{1}{N}\sum_{i=1}^{N}\sqrt{(N-1)\,\mathrm{Var}_{j}\!\left(J_{ij}\right)},
\label{eqS:sJ}
\end{equation}
and the reference schedule endpoints are
\begin{equation}
I_0^{\min}=0.1/s_J,
\qquad
I_0^{\max}=10/s_J.
\label{eqS:I0schedule}
\end{equation}

The numerical parameters used throughout the simulations are summarized in Table~\ref{tabS:settings}.

\begin{table*}[t]
\caption{Simulation settings used throughout the study.}
\label{tabS:settings}
\centering
\begin{tabular}{ll}
\toprule
Item & Value / Description \\
\midrule
Graph instances & G1, G6, G11, G35, G39, G58, G63, G64 and Toy-1--Toy-7 \\
Update scheme & tick-random synchronous, $p=1/c$ \\
Anneal schedule & constant $I_0=I_0^{\max}$ in the fixed-$I_0$ sweeps \\
Simulation duration & 200 ns, $\Delta t=5$ ns (40 ticks) \\
$c$ sweep & 1.0 to 5.0 in steps of 0.05 \\
Oscillation criterion & $C(1)<0.5$ with energy-based qualitative confirmation \\
Repeats & 5 random seeds per configuration \\
\bottomrule
\end{tabular}
\end{table*}

\section{Graph Instances and Metadata}

 In Table~\ref{tabS:instances}, Erd\H{o}s--R\'enyi (ER) denotes the random graph instances used as illustrative baselines. 
 Here AF and FM denote antiferromagnetic and ferromagnetic couplings, respectively, and C8 denotes the eight-node cycle graph. 

\begin{table}[t]
\caption{G-set and illustrative graph instances used in the study.}
\label{tabS:instances}
\centering
\begin{tabular}{lcccc}
\toprule
Instance & Nodes & Edges & Weight value(s) & Weight type \\
\midrule
G1  & 800  & 19176 & $+1$ & random \\
G6  & 800  & 19176 & $\pm 1$ & random \\
G11 & 800  & 1600  & $\pm 1$ & toroidal \\
G35 & 2000 & 11778 & $+1$ & planar \\
G39 & 2000 & 11778 & $\pm 1$ & planar \\
G58 & 5000 & 29570 & $+1$ & planar \\
G63 & 7000 & 41459 & $+1$ & planar \\
G64 & 7000 & 41459 & $\pm 1$ & planar \\
\midrule
Toy-1 & 4  & 4  & $+1$ & AF ring \\
Toy-2 & 4  & 6  & $-1$ & FM complete \\
Toy-3 & 6  & 7  & $+1$ & AF triangles + bridge \\
Toy-4 & 8  & 8  & $+1$ & AF ring (C8) \\
Toy-5 & 16 & 24 & $-1$ & FM $4\times 4$ lattice \\
Toy-6a & 32 & 93 & $\pm 1$ & ER random ($k\approx 6$) \\
Toy-6b & 32 & 107 & $\pm 1$ & ER random ($k\approx 6$) \\
Toy-7 & 64 & 200 & $\pm 1$ & ER random ($k\approx 6$) \\
\bottomrule
\end{tabular}
\end{table}

 The small illustrative graphs illustrate the oscillation mechanism and its suppression under increasing $c$, while the larger benchmark and random-graph instances provide the more relevant regime for quantitative comparison with the mean-field and Gaussian theories. 

\section{Additional Diagnostics}

\subsection{Representative Small Systems}

 Representative systems with $N<16$ are included to show where the Gaussian closure is only qualitative. In these very small systems, local-field distributions are discrete and self-averaging is weak, so deviations between the Gaussian prediction and simulation should be interpreted as finite-size closure effects rather than as a failure of the finite-time observability framework.

Representative results for systems with $N<16$ are shown in Fig.~\ref{figS:small_system_diagnostic}, illustrating the qualitative agreement and the limitations of the Gaussian closure.
Here $\alpha_{\rm Gauss}$ denotes the Gaussian effective gain $\alpha_{\mathrm{eff}}(I_0)$ defined in Sec.~VII, whereas $\alpha_{\rm exact}$ denotes the corresponding effective gain evaluated by exact enumeration over all spin configurations for these small systems.
The small deviation for Toy-2 is mainly structural rather than a consequence of the negative graph weights: in the four-node complete graph, the enumerated local fields avoid $h=0$, making the exact gain nearly zero, while the Gaussian gain is also small at the same $I_0$.

\begin{figure}[t]
\centering
\includegraphics[width=0.98\linewidth]{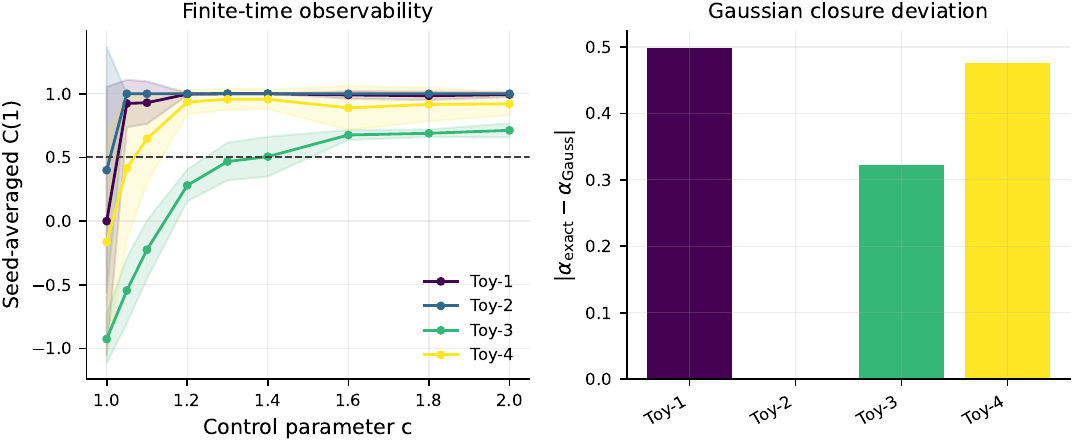}
\caption{Representative small-system diagnostic for the finite-time observability and Gaussian-closure comparison. The figure uses existing revision data and summarizes the qualitative behavior of $N<16$ systems without introducing additional simulations. }
\label{figS:small_system_diagnostic}
\end{figure}

\subsection{Graph-Property Comparison}

 Table~\ref{tabS:graph_properties} gives a compact comparison for the four graph instances highlighted in the main text. The comparison supports the conclusion that sparsity alone does not determine stabilization behavior; graphs with similar density can have different signed or weighted interaction structure and different normalized spectral scale. 
 Here the positive fraction denotes the fraction of positive graph weights $w_{ij}>0$, and $\rho(J)$ denotes the spectral radius of the normalized interaction matrix used in the finite-time analysis. 

\begin{table*}[t]
\caption{Compact graph-property comparison for representative G-set instances. Here $\rho(J)$ denotes the normalized spectral radius of the interaction matrix used in the finite-time analysis, and $c^\star$ is the simulation threshold from the one-step autocorrelation criterion. }
\label{tabS:graph_properties}
\centering
\begin{tabular}{lrrrrrrr}
\toprule
Graph & $N$ & Edges & Density & Avg. degree & Positive fraction & $\rho(J)$ & $c^\star$ \\
\midrule
G1  & 800  & 19176 & 0.0600  & 47.94 & 1.000 & 7.28 & 2.05 \\
G6  & 800  & 19176 & 0.0600  & 47.94 & 0.504 & 2.01 & 1.10 \\
G39 & 2000 & 11778 & 0.00589 & 11.78 & 0.501 & 4.71 & 1.10 \\
G58 & 5000 & 29570 & 0.00237 & 11.83 & 1.000 & 9.37 & 1.60 \\
\bottomrule
\end{tabular}
\end{table*}

\subsection{Autocorrelation and Optimization Diagnostics}

 Existing autocorrelation diagnostics show that $C(1)$ is a complementary stabilization criterion: it identifies the delayed stabilization of G1 and G58 but does not reproduce every local oscillation of the normalized second-difference amplitude. The second-difference amplitude and the normalized Max-Cut score also quantify different aspects of the same transition, namely residual period-2 energy motion and optimization quality, so a feature-by-feature correspondence is not expected.

The complementary behavior of the one-step autocorrelation criterion is illustrated in Fig.~\ref{figS:autocorrelation_diagnostic}.

\begin{figure}[t]
\centering
\includegraphics[width=0.98\linewidth]{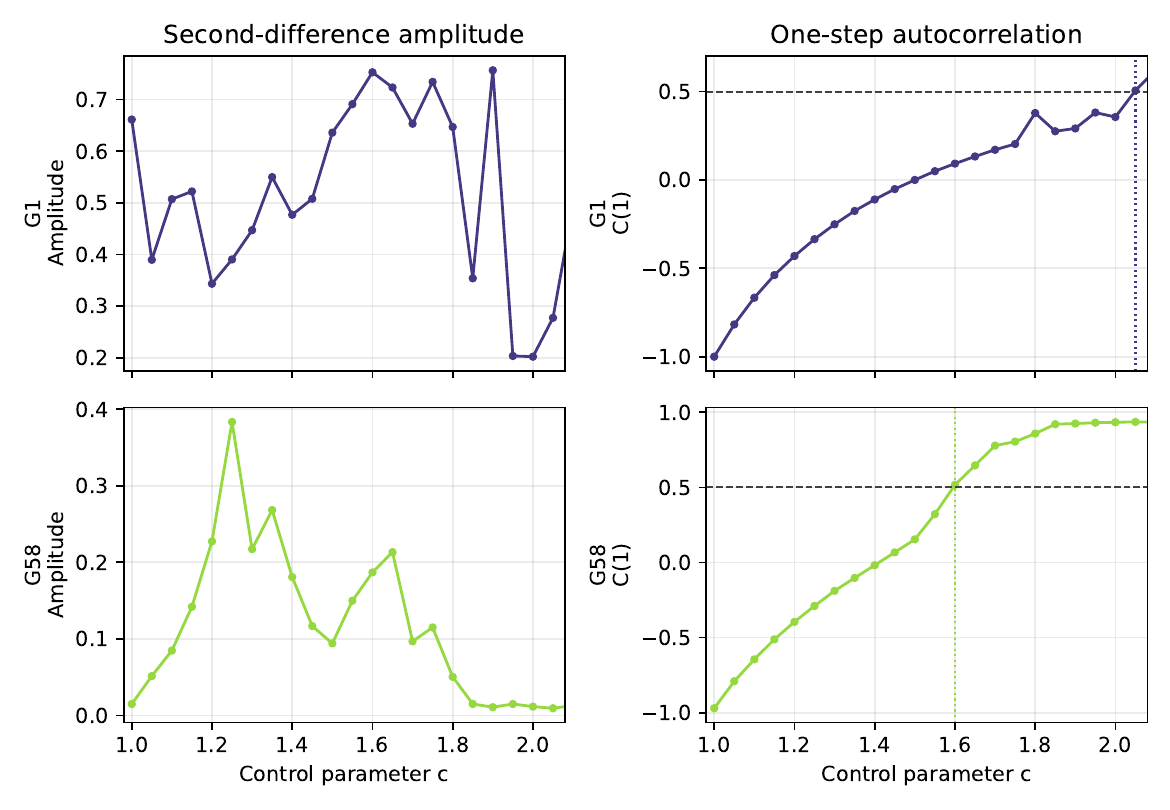}
\caption{Diagnostic comparison of the one-step autocorrelation criterion and the normalized second-difference amplitude using existing c-sweep data. The autocorrelation provides a smoother stabilization criterion and is not expected to reproduce local amplitude oscillations one-for-one. }
\label{figS:autocorrelation_diagnostic}
\end{figure}

\bibliographystyle{apsrev4-2}
	\bibliography{main,supplement_reference}
	
\end{document}